\documentclass[11pt]{article}
\usepackage{graphicx}

\usepackage{amsmath,amsfonts,amssymb,amsthm}
\usepackage{bbm,mathrsfs, bm, mathtools}

\title{Small faces in stationary Poisson hyperplane tessellations}
\author{Rolf Schneider}
\date{}
\newcommand{\R}{{\mathbb R}}

\newcommand{\K}{{\mathcal K}}

\newcommand{\bP}{{\mathbb P}}
\newcommand{\cP}{{\mathcal P}}

\newcommand{\N}{{\mathbb N}}

\newcommand{\Ha}{\mathcal{H}}

\newcommand{\B}{\mathcal{B}}
\newcommand{\D}{{\rm d}}

\newcommand{\bE}{{\mathbb E}\,}

\newcommand{\bQ}{{\mathbb Q}}

  \renewcommand{\dim}{{\rm dim}\,}

\newtheorem{theorem}{Theorem}%[section]
\newtheorem{lemma}{Lemma}%[section]
\newtheorem{definition}{Definition}%[section]

\begin{document}
\maketitle

\begin{abstract}
We consider the tessellation induced by a stationary Poisson hyperplane process in $d$-dimensional Euclidean space. Under a suitable assumption on the directional distribution, and measuring the $k$-faces of the tessellation by a suitable size functional, we determine a limit distribution for the shape of the typical $k$-face, under the condition of small size and this tending to zero. The limit distribution is concentrated on simplices. This extends a result of Gilles Bonnet.\\[1mm]
\end{abstract}

\section{Introduction}\label{sec1}

A stationary Poisson hyperplane process in $\R^d$ gives rise to a tessellation of $\R^d$ into convex polytopes. The polytopes in such a Poisson hyperplane tessellation pose a considerable number of geometrically interesting questions. First, one may notice that generally the variety of the appearing cells (the $d$-dimensional polytopes) is very rich: under mild assumptions on the directional distribution of the hyperplane process (which are satisfied, for example, in the isotropic case), the translates of the cells in the tessellation are a.s. dense (with respect to the Hausdorff metric) in the space of convex bodies in $\R^d$. Further, with probability one, every combinatorial type of a simple $d$-polytope appears infinitely often among the cells of the tessellation (see \cite{RS16}). The latter result can be strengthened: from `infinitely often' to `with positive density', see \cite{Sch18}. A different type of questions comes up if one asks for average cells, formalized by the notions of the `zero cell' and the `typical cell'. A well-known question of D.G. Kendall concerned an isotropic Poisson line tessellation in the plane and asked whether the conditional law for the shape of the zero cell, given its area, converges weakly, if the area tends to infinity, to the degenerate law concentrated at the circular shape. Proofs were given by Kovalenko \cite{Kov97, Kov99}. This led to detailed investigations into the relations between large size and shape, for zero cells and typical cells, in higher dimensions, for non-isotropic tessellations, and for general (axiomatized) notions of size. We refer to \cite{HRS04a, HRS04b, HS07a, HS07b}. The extension from cells to lower-dimensional faces revealed new aspects, since the shape of a large face is influenced not only by its size, but also by its direction; see \cite{HS10, HS11}.

Small cells, on the other side, have been studied with less intensity, so far. If one looks at a sufficiently large simulation of a stationary, isotropic Poisson line process in the plane, one will notice that the very small cells are mostly triangles, but, of course, of varying shapes and orientations. In his heuristic approach to Kendall's problem, Miles \cite{Mil95} also briefly considered small cells. We quote his conclusion: ``Thus it may be said that the `small' polygons of ${\mathscr P}$ are triangles of random shape, the corresponding shape distribution being immediately dependent on the conditioning characteristic determining `smallness'. This contrasts strongly with the `large' polygons (\dots), for which shape distributions are characteristically degenerate, even though evidently dependent on the particular conditioning variable.'' This concerns the isotropic case. Poisson line processes with only two directions of the lines, where no triangles can be produced, were investigated by Beermann, Redenbach, and Th\"ale \cite{BRT14}. They showed that the asymptotic shape of cells with small area is degenerate, whereas the asymptotic shape of cells with small perimeter is indeterminate. The first investigation of small cells in higher dimensions is due to Bonnet \cite{Bon17}. Under various assumptions on the directional distribution of the hyperplane process, he found different asymptotic behavior for small cells, including results on the speed of convergence. If the spherical directional distribution of the hyperplane process is absolutely continuous with respect to spherical Lebesgue measure, he established that the distribution of the shape of the typical cell, under the condition of small size and this size tending to zero, converges to a distribution on the space of simplex shapes, which was represented explicitly in terms of the directional distribution and the employed size measurement. The present note extends this result in two directions. First, the assumption of absolute continuity is weakened: the optimal assumption seems to be that the spherical directional distribution assigns measure zero to each great subsphere. Second, and more important, we consider also lower-dimensional faces and extend Bonnet's result to the typical $k$-face of the tessellation, for $k\in\{1,\dots,d\}$. This requires, first, to generalize a representation of the distribution of the typical cell (Theorem 10.4.6 in \cite{SW08}, in the isotropic case), involving the inball center as a center function, to typical $k$-faces and to more general directional distributions. The result is Theorem \ref{T4.1} in Section \ref{sec4}. Then Theorem \ref{T5.1} in Section \ref{sec5} extends Bonnet's result to the typical $k$-face.

\section{Preliminaries}\label{sec2}

We fix some notation. We work in $\R^d$ ($d\ge 2$), the $d$-dimensional real vector space with origin $o$, scalar product $\langle\cdot\,,\cdot\rangle$ and induced norm $\|\cdot\|$. The unit ball of $\R^d$ is denoted by $B^d$, the unit sphere by $S^{d-1}$, and the Lebesgue measure on $\R^d$ by $\lambda_d$. By $\Ha$ we denote the space of hyperplanes in $\R^d$, with its usual topology. For $K\subset \R^d$, we denote by
$$ [K]_\Ha= \{H\in\Ha:H\cap K\not=\emptyset\}$$
the set of hyperplanes meeting $K$. Every hyperplane in $\Ha$ has a representation 
$$ H(u,\tau)= \{x\in\R^d: \langle x,u\rangle=\tau\}$$
with $u\in S^{d-1}$ and $\tau\in\R$, and we have $H(u,\tau)=H(u',\tau')$ if and only if $(u,\tau)=(\epsilon u',\epsilon\tau')$ with $\epsilon\in\{-1,1\}$. The vectors $\pm u$ are the unit normal vectors of the hyperplane $H(u,\tau)$, and usually one of them is chosen arbitrarily as `the' normal vector of the hyperplane. The set
$$ H^-(u,\tau) = \{x\in\R^d:\langle x,u\rangle\le \tau\}$$
is one of the closed halfspaces bounded by $H(u,\tau)$. If $H$ is a hyperplane and $x\in\R^d$ a point not in $H$, we denote by $H^-_x$ the closed halfspace bounded by $H$ that contains $x$. 

The space $\K$ of convex bodies (nonempty, compact, convex subsets) of $\R^d$ is equipped with the Hausdorff metric. The subspace $\cP$ of polytopes is a Borel set in $\K$. Generally for a topological space $T$, we denote by $\B(T)$ the $\sigma$-algebra of Borel sets in $T$. 

In the following, $\widehat X$ will always be a non-degenerate stationary Poisson hyperplane process in $\R^d$. We refer to \cite{SW08}, in particular Sections 3.2, 4.4. and 10.3, for the basic notions and fundamental facts about such processes. We denote the underlying probability by $\bP$, and expectation by $\bE$. The intensity measure $\widehat\Theta$ of $\widehat X$, defined by
$$ \widehat\Theta (A) := \bE \widehat X(A),\quad A\in\B(\Ha),$$
has a decomposition
\begin{equation}\label{2.1}
\int_{\Ha} f\,\D\widehat\Theta = \widehat\gamma\int_{S^{d-1}} \int_{\infty}^\infty f(H(u,\tau))\,\D\tau\,\varphi(\D u)
\end{equation}
for every nonnegative, measurable function $f$ on $\Ha$; see \cite[(4.33)]{SW08}. Here $\widehat\gamma>0$ is the {\em intensity} of $\widehat X$ and $\varphi$ is an even Borel probability measure on $S^{d-1}$, called the {\em spherical directional distribution} of $\widehat X$.

For a convex body $K\in\K$, the distribution of the number of hyperplanes of $\widehat X$ hitting $K$ is given by
$$ \bP\left(\widehat X([K]_\Ha)=n\right) =e^{-\widehat\Theta([K]_\Ha)} \frac{\widehat\Theta([K]_\Ha)^n}{n!},\quad n\in\N_0.$$
It follows from (\ref{2.1}) that 
$$ \widehat\Theta([K]_\Ha) =\widehat\gamma \int_{S^{d-1}}\int_{-\infty}^\infty {\mathbbm 1}\{H(u,\tau)\cap K\not=\emptyset\}\,\D\tau\,\varphi(\D u) = 2\widehat\gamma \Phi(K)$$
with
$$ \Phi(K) =\int_{S^{d-1}} h(K,u)\,\varphi(\D u),$$
where $h(K,u)= \max\{\langle x,u\rangle:x\in K\}$ defines the support function $h(K,\cdot)$ of $K$. (By ${\mathbbm 1}$ we always denote an indicator function.) For obvious reasons, we call $\Phi$ the {\em hitting functional} of $\widehat X$.

For convenience, we often identify a simple counting measure with its support, so we write $H\in\widehat X(\omega)$ for $\widehat X(\omega)(\{H\})=1$. In particular, Campbell's theorem for $\widehat X$ (and for $f$ as above) is written in the form
\begin{equation}\label{2.2}
\bE \sum_{H\in\widehat X} f(H) = \int_{\Ha} f\,\D\widehat\Theta.
\end{equation}

The following notions of general position play a role for hyperplane processes.

\begin{definition}\label{D2.1}
A system of hyperplanes in $\R^d$ is in {\em translationally general position} if every $k$-dimensional plane of $\R^d$ is contained in at most $d-k$ hyperplanes of the system, for $k=0,\dots,d-1$.

A system of hyperplanes in $\R^d$ is in {\em directionally general position} if any $d$ of the hyperplanes of the system have linearly independent normal vectors.

Translationally general position together with directionally general position is called {\em general position}.
\end{definition}

While the hyperplanes of $\widehat X$ a.s. have the first property, they have the second property only under an assumption on its spherical directional distribution.

\begin{definition}\label{D2.2}
A measure on $S^{d-1}$ is called {\em subspace-free} if it is zero on each great subsphere of $S^{d-1}$.
\end{definition}

In the proof below, and later, we denote by $\widehat X^m_{\not=}$ the set of all ordererd $m$-tuples of pairwise distinct hyperplanes from $\widehat X$. 

\begin{lemma}\label{L2.1}
Let $\widehat X$ be nondegenerate, stationary Poisson hyperplane process in $\R^d$.\\[1mm]
$\rm (a)$ 
The hyperplanes of $\widehat X$ are a.s. in translational general position.\\[1mm]
$\rm (b)$ If the spherical directional distribution $\varphi$ of $\widehat X$ is subspace-free, then a.s. the hyperplanes of $\widehat X$ are in directionally general position.
\end{lemma}

\begin{proof}
Let $m\in\N$ and $A\in\B(\Ha^m)$. We apply \cite{SW08}, Theorem 3.1.3 and Corollary 3.2.4, and then (\ref{2.1}), to obtain
\begin{eqnarray}\label{3.2.0}
&& \bE \sum_{(H_1,\dots,H_m)\in \widehat X^m_{\not=}} {\mathbbm 1}_A(H_1,\dots,H_m)\\
&& = \int_{\Ha^m} {\mathbbm 1}_A(H_1,\dots,H_m)\,\widehat\Theta^m(\D(H_1,\dots,H_m)) \nonumber\\
&& = \widehat\gamma^m \int_{(S^{d-1})^m} \int_{-\infty}^\infty \cdots  \int_{-\infty}^\infty {\mathbbm 1}_A(H(u_1,\tau_1),\dots,H(u_m,\tau_m))\,\D\tau_1\cdots\D\tau_m\,\varphi^m(\D(u_1,\dots,u_m)). \nonumber
\end{eqnarray}

\vspace{2mm}

\noindent(a) Let $k\in \{0,\dots,d-1\}$ be given. Define $A\in\B(\Ha^{d-k+1})$ by
$$ A:= \{(H_1,\dots,H_{d-k+1})\in\Ha^{d-k+1}: \dim(H_1\cap\dots\cap H_{d-k+1})=k\}.$$
Suppose that, for fixed $u_1,\dots,u_{d-k+1}$ and suitable $\tau_1,\dots,\tau_{d-k+1}$,
$$ \dim(H(u_1,\tau_1)\cap\dots\cap H(u_{d-k+1},\tau_{d-k+1})) = k.$$
Then we can assume, after renumbering, that
$$ \dim(H(u_1,\tau_1)\cap\dots\cap H(u_{d-k},\tau_{d-k})) = k,$$
and we have
$$ H(u_1,\tau_1)\cap\dots\cap H(u_{d-k},\tau_{d-k})\cap H(u_{d-k+1},\tau) =\emptyset$$
for all $\tau$ except one value. Therefore
$$ \int_{-\infty}^\infty {\mathbbm 1}_A(H(u_1,\tau_1),\dots,H(u_{d-k},\tau_{d-k}), H(u_{d-k+1},\tau))\,\D\tau=0.$$
Relation (\ref{3.2.0}) (with $m=d-k+1$) gives the assertion. 

\vspace{2mm}

\noindent(b) Suppose that $\varphi$ is subspace-free. Let $A$ be the set of all $(d+1)$-tuples $(H_1,\dots,H_{d+1})\in\Ha^{d+1}$ with linearly dependent normal vectors. Then $A=\bigcup_{i=1}^{d+1} A_i$, where $A_i$ is the set of all $(d+1)$-tuples $(H_1,\dots,H_{d+1})\in\Ha^{d+1}$ for which the normal vector of $H_i$ is in the linear hull of the normal vectors of the remaining hyperplanes. For fixed $u_1,\dots,u_{d},\tau_1,\dots,\tau_{d+1}$ we have
$$ \int_{S^{d-1}} {\mathbbm 1}_{A_{d+1}}(H(u_1,\tau_1),\dots, H(u_{d},\tau_{d}),H(u_{d+1},\tau_{d+1}))\,\varphi(\D u_{d+1})=0,$$
since the integrand is zero unless $u_{d+1}$ lies in some fixed $(d-1)$-dimensional great subsphere; similarly for $A_1,\dots,A_{d}$. Now relation (\ref{3.2.0}) for $m=d+1$, together with Fubini's theorem, gives the assertion. 
\end{proof}

\vspace{3mm}

The stationary Poisson hyperplane process $\widehat X$ induces a tessellation $X$ of $\R^d$ into convex polytopes. In particular, for each $k\in\{0,\dots,d\}$ it induces the particle process $X^{(k)}$ of its $k$-faces (thus, $X^{(d)}=X$ is the process of its `cells'). 

Generally, let $Y$ be a stationary particle process in a Borel subset $\K'$ of $\K$. A {\em center function} on $\K'$ is a measurable map $c:\K'\to\R^d$ satisfying $c(K+t)=c(K)+t$ for $t\in\R^d$ and $c(K)=o \Rightarrow c(\alpha K)=o$  if $\alpha\ge 0$, for all $K\in\K'$. An often used center function is the circumcenter (the center of the smallest ball containing the set), but later we shall need a different one, not defined on all of $\K$. We define $\K'_c:= \{K\in \K':c(K)=o\}$. If the intensity measure $\Theta=\bE Y$ satisfies $\Theta\not=0$ and $\Theta(\{K\in\K': 
K'\cap C\not=\emptyset\})<\infty$ for every compact subset $C\subset\R^d$, then there are a unique number $\gamma>0$ and a unique probability measure $\bQ$ on $\K'_c$ such that
\begin{equation}\label{2.3}
\int_{\K'}f\,\D\Theta = \gamma\int_{\K_c'}\int_{\R^d} f(K+x)\,\lambda_d(\D x)\,\bQ(\D K)
\end{equation}
for every nonnegative, measurable function $f$ on $\K'$. We refer to \cite[Section 4.1]{SW08}, but point out that the definition of a center function has slightly been changed. The measure $\bQ$, called the {\em grain distribution} of $Y$, depends on the choice of the center function, although this is not shown in the notation. The intensity $\gamma$ is independent of the choice of the center function, as follows, e.g., from \cite{SW08}, Theorem 4.1.3(b) with $\varphi\equiv 1$. Campbell's theorem now reads 
\begin{equation}\label{2.4}
\bE \sum_{K\in Y} f(K) = \int_{\K'} f\,\D\Theta.
\end{equation}
Let $A\in \B(\K')$ and $B\in\B(\R^d)$ with $0<\lambda_d(B)<\infty$. Applying (\ref{2.3}) and (\ref{2.4}) with $f(K)= {\mathbbm 1}_A(K-c(K)){\mathbbm 1}_B(c(K))$, we get
\begin{equation}\label{2.5} 
\gamma\bQ(A) =\frac{1}{\lambda_d(B)}\, \bE \sum_{K\in X,\,c(K)\in B} {\mathbbm 1}_A(K-c(K)).
\end{equation}

For the process $X^{(k)}$ of $k$-faces of the tessellation $X$ induced by $\widehat X$, we denote the intensity by $\gamma^{(k)}$ and the intensity measure by $\bQ^{(k)}$. This intensity measure satisfies the assumption made above (as follows, e.g., from \cite[Thm. 10.3.4]{SW08}), hence the previous results can be applied. By $\cP_k\subset \K$ we denote the set of $k$-dimensional polytopes; then $\cP_k$ is a Borel subset of $\K$. Since $X^{(k)}$ is a particle process in $\cP_k$, for its description we need only center functions on $\cP_k$. If such a center function is chosen, we denote by $Z^{(k)}$ the {\em typical $k$-face} of $X$, which is defined as a random polytope with distribution $\bQ^{(k)}$. Thus, for $A\in\B(\cP_k)$ and $B\in\B(\R^d)$ with $\lambda_d(B)=1$, we have
\begin{equation}\label{2.6}
\gamma^{(k)}\bQ^{(k)}(A) =\bE\sum_{K\in X^{(k)}} {\mathbbm 1}_A(K-c(K)){\mathbbm 1}_B(c(K)).
\end{equation}

\section{The incenter}\label{sec3}

Later arguments require to base the investigation of the typical $k$-face on the incenter as a center function. The incenter of a $k$-dimensional convex body is the center of its largest inscribed $k$-dimensional ball, if that is unique. Since the latter is not always the case, some precautions are necessary. The inradius of a $k$-dimensional convex body is the radius of a largest $k$-dimensional ball contained in the body.

\begin{lemma}\label{L3.1}
Let $P\in\cP$ be a $d$-polytope with the property that its facet hyperplanes are in directionally general position. Then $P$ has a unique inball.
\end{lemma}

\begin{proof}
Let $B$ be an inball of $P$. Let $u_1,\dots,u_k$ be the outer unit normal vectors of the facets of $P$ that touch $B$, and let $C$ be the positive hull of these vectors. If the closed convex cone $C$ does not contain a linear subspace of positive dimension, then there is a vector $t$ in the interior of the polar cone of $C$. This vector satisfies $\langle t,u_i\rangle<0$ for $i=1,\dots,k$ and hence $B+\varepsilon t \subset{\rm int}\,P$ for some $\varepsilon>0$, a contradiction. Thus, $C$ contains a subspace $L$ of dimension $k>0$. Since $L$ is positively spanned by normal vectors of $P$, it contains at least $k+1$ such vectors, and these are linearly dependent. This shows that $k=d$. But if the facets of $P$ touching $B$ have normal vectors which positively span $\R^d$, then these facets fix the ball $B$ against translations, hence it is the unique inball of $P$.
\end{proof}

The notions of general position immediately carry over to $(k-1)$-flats in a given $k$-dimensional affine subspace $L$, since after choosing an origin in $L$, we can view $L$ as a $k$-dimensional  real vector space. 

Let ${\bf u}=(u_1,\dots,u_{d+1})$ be a $(d+1)$-tuple of unit vectors in general position, that is, any $d$ of the vectors are linearly independent. For given $k\in \{1,\dots,d-1\}$, we define the $k$-dimensional subspace
$$ L_{\bf u} = \left({\rm lin}\{u_{k+2},\dots,u_{d+1}\}\right)^\perp.$$
We have orthogonal decompositions
$$ u_j=v_j+w_j\quad\mbox{with }v_j\in L_{\bf u},\,w_j\in L_{\bf u}^\perp,\; j=1,\dots,k+1.$$
Here $v_j\not= o$, since otherwise $u_j\in L_{\bf u}^\perp={\rm lin}\{u_{k+2},\dots,u_{d+1}\}$, a contradiction. We claim that any $k$ vectors of $v_1,\dots,v_{k+1}$ are linearly independent. Otherwise, some $k$ of them, say $v_1,\dots,v_k$, span a proper subspace $E$ of $L_{\bf u}$. But then $u_1,\dots,u_k,u_{k+2},\dots,u_{d+1}$ span only a proper subspace of $\R^d$, which is a contradiction.

Now we can prove the following lemma.

\begin{lemma}\label{L3.2}
Let $H_1,\dots,H_{d+1}$ be hyperplanes in general position. Let $k\in\{1,\dots,d-1\}$ and let $L:=\bigcap_{i=k+2}^{d+1} H_i$. Let $h_j:= H_j\cap L$ for $j=1,\dots,k+1$. Then $h_1,\dots,h_{k+1}$ are in general position with respect to $L$.
\end{lemma}

\begin{proof}
Let $u_i$ be a unit normal vector of $H_i$, for $i=1,\dots,d+1$. Then ${\bf u}=(u_1,\dots,u_{d+1})$ is in general position, and we define $L_{\bf u}$ and the vectors $v_1,\dots,v_{k+1}$ as above. The affine subspace $L$ is a translate of $L_{\bf u}$. Let $x,y\in h_j$. Then $\langle x-y,u_j\rangle =0$ and $\langle x-y,w_j\rangle =0$, hence $\langle x-y,v_j\rangle =0$. It follows that $v_j$ is a normal vector of $h_j$ with respect to $L$. Thus, $h_1,\dots,h_{k+1}$ are in directionally general position. They are also in translationally general position, since they do not have a common point, because $H_1,\dots,H_{d+1}$ do not have a common point. 
\end{proof}

The integral transform of the subsequent lemma is an extension of Theorem 7.3.2 in \cite{SW08} (and also a correction---a factor $2^{d+1}$ was missing there). The extension consists in admitting more general measures (not necessarily motion invariant) and using incenter and inradius of also lower-dimensional simplices. The lemma requires some preparations.

Let $H_1,\dots,H_{d+1}$ be hyperplanes in general position, and let $k\in\{1,\dots,d-1\}$. Let $u_i$ be a unit normal vector of $H_i$, for $i=1,\dots,d+1$. We define $L_{\bf u}$, $L$, $h_1,\dots,h_{k+1}$, and the vectors $v_1,\dots,v_{k+1}$ as above. Thus, the $k$-flat $L$, the $(k-1)$-flats $h_1,\dots,h_{k+1}$ and their normal vectors $v_1,\dots,v_{k+1}$ in $L_{\bf u}$ depend on $H_1,\dots,H_{d+1}$ and on the number $k$.

Since $h_1,\dots,h_{k+1}$ are in general position in $L$, they determine a unique $k$-simplex $\triangle(H_1,\dots,H_{d+1})\subset L$ such that $h_1,\dots,h_{k+1}$ are the facet hyperplanes of $\triangle$  (in $L$). We denote by $z(H_1,\dots,H_{d+1})$ its incenter and by $r(H_1,\dots,H_{d+1})$ its inradius. The dependence on $k$ is not shown in the notation, since $k$ is fixed.

We denote by ${\sf P}_k\subset (S^{d-1})^{d+1}$ the set of all $(d+1)$-tuples $(u_1,\dots,u_{d+1})$ in general position such that the vectors $v_1,\dots,v_{k+1}$ positively span the subspace $L_{\bf u}$.

For ${\bf u}=(u_1,\dots,u_{d+1})\in{\sf P}_k$  and for $z\in\R^d$ and $r\in\R^+$, we define
$$ t_j({\bf u},z,r):= \left\{\begin{array}{ll} \langle z,u_j\rangle + r\|u_j|L_{\bf u}\| & \mbox{if } j\in\{1,\dots,k+1\},\\[2mm]
\langle z, u_j\rangle  & \mbox{if } j\in\{k+2,\dots,d+1\},\end{array}\right.$$
where $u_j|L_{\bf u}$ is the image of $u_j$ under orthogonal projection to $L_{\bf u}$. The mapping 
$$ (z,r) \mapsto (t_1({\bf u},z,r),\dots,t_{d+1}({\bf u},z,r))$$ 
has Jacobian 
$$ D_k({\bf u}):= \det\left(\begin{array}{cccc} \langle u_1,e_1\rangle & \cdots & \langle u_1,e_d\rangle & \|u_1|L_{\bf u}\|\\ 
\vdots & & \vdots & \vdots\\
\langle u_{k+1},e_1\rangle & \cdots & \langle u_{k+1},e_d\rangle & \|u_{k+1}|L_{\bf u}\|\\
\langle u_{k+2},e_1\rangle & \cdots & \langle u_{k+2},e_d\rangle & 0\\
\vdots & & \vdots & \vdots\\
\langle u_{d+1},e_1\rangle & \cdots & \langle u_{d+1},e_d\rangle & 0
\end{array}\right),
$$
where $(e_1,\dots,e_d)$ is an orthonormal basis of $\R^d$. 

\begin{lemma}\label{L3.3}
Suppose that the spherical directional distribution $\varphi$ of $\widehat X$ is subspace-free. Let $k\in \{1,\dots,d\}$. If $f:\Ha^{d+1}\to\R$ is a nonnegative, measurable function, then
\begin{eqnarray*}
&& \int_{\Ha^{d+1}} f\,\D\widehat\Theta^{d+1}\\
&&=2^{k+1}\widehat\gamma^{d+1} \int_{{\sf P}_k} \int_{\R^d} \int_0^\infty f(H(u_1,t_1({\bf u},z,r)),\dots,H(u_{d+1}, t_{d+1}({\bf u},z,r)))\\
&&\hspace{4mm}\times\;\D r\,\lambda_d(z)D_k({\bf u})\,\varphi^{d+1}(\D{\bf u}).
\end{eqnarray*}
\end{lemma}

\begin{proof}
In the following proof, we assume that $k\le d-1$. The proof for the case $k=d$ requires only the replacement of $L_{\bf u}$, $L$ by $\R^d$, of $v_i$ by $u_i$, and the obvious modifications. We have $D_d({\bf u}) = d!\Delta({\bf u})$, where $\Delta({\bf u})$ is the volume of the convex hull of $u_1,\dots,u_{d+1}$.

For $(u_i,\tau_i)\in S^{d-1}\times\R$, write
$$ f'((u_1,\tau_1),\dots,(u_{d+1},\tau_{d+1})) := f(H(u_1,\tau_1),\dots,H(u_{d+1},\tau_{d+1})).$$
According to (\ref{2.1}), we can write
$$ \int_{\Ha^{d+1}} f\,\D\widehat\Theta^{d+1} = \widehat\gamma^{d+1} \int_{(S^{d-1}\times\R)^{d+1}} f'\,\D(\varphi\otimes\lambda_1)^{d+1},$$
where $\lambda_1$ denotes Lebesgue measure on $\R$. Since $\varphi$ is subspace-free, $(\varphi\otimes\lambda_1)^{d+1}$-almost all $(d+1)$-tuples $((u_1,\tau_1),\dots,(u_{d+1},\tau_{d+1}))\in (S^{d-1}\times\R)^{d+1}$ have the property that the hyperplanes $H_i=H(u_i,\tau_i)$, $i=1,\dots,d+1$, determine a unique $k$-simplex $\triangle(H_1,\dots,H_{d+1})$ as above. The vectors $v_1,\dots,v_{k+1}$, defined as above, are the normal vectors of the facets of $\triangle$, either inner or outer normal vectors. We define $A$ as the set of all $((u_1,\tau_1),\dots,(u_{d+1},\tau_{d+1}))\in (S^{d-1}\times\R)^{d+1}$ for which the hyperplanes $H(u_1,\tau_1),\dots,H(u_{d+1},\tau_{d+1})$ are in general position and determine a $k$-simplex $\triangle$ that has the vectors $v_1,\dots,v_{k+1}$ as {\bf outer} normal vectors. For ${\mathcal E}=(\epsilon_1,\dots,\epsilon_{k+1})\in\{-1,1\}^{k+1}$, let $T_{\mathcal E}: (S^{d-1}\times\R)^{d+1}\to (S^{d-1}\times\R)^{d+1}$ be defined by
\begin{eqnarray*}
&&T_{\mathcal E}((u_1,\tau_1),\dots,(u_{d+1},\tau_{d+1}))\\
&& := ((\epsilon_1 u_1,\epsilon_1 \tau_1),\dots,(\epsilon_{k+1} u_{k+1},\epsilon_{k+1} \tau_{k+1}),(u_{k+2},\tau_{k+2}),\dots,(u_{d+1},\tau_{d+1})).
\end{eqnarray*}
Then the sets $T_{\mathcal E}(A)$, ${\mathcal E}\in \{-1,1\}^{k+1}$, are pairwise disjoint and cover $(S^{d-1}\times\R)^{d+1}$ up to a set of $(\varphi\otimes\lambda_1)$-measure zero. We have $f'\circ T_{\mathcal E}=f'$. Since the pushforward of $(\varphi\otimes\lambda_1)^{d+1}$ under $T_{\mathcal E}$ is the same measure (since $\varphi$ and $\lambda_1$ are even), we have
$$
\int_{(S^{d-1}\times\R)^{d+1}} f'\,\D(\varphi\otimes\lambda_1)^{d+1} = \sum_{\mathcal E} \int_{T_{\mathcal E}(A)} f'\,\D(\varphi\otimes\lambda_1)^{d+1}
 = 2^{k+1} \int_A f'\,\D(\varphi\otimes\lambda_1)^{d+1}.
$$
This gives 
\begin{eqnarray*}
\int_{\Ha^{d+1}} f\,\D\widehat\Theta^{d+1} &=&
\widehat\gamma^{d+1}  \int_{(S^{d-1}\times\R)^{d+1}} f'\,\D(\varphi\otimes\lambda_1)^{d+1} =  2^{k+1}\widehat\gamma^{d+1} \int_A f'\,\D(\varphi\otimes\lambda_1)^{d+1}\\
&=& 2^{k+1}\widehat\gamma^{d+1} \int_{{\sf P}_k}\int_{\R^{d+1}} f(H(u_1,\tau_1),\dots,H(u_{d+1},\tau_{d+1}))\\
&& \times\; {\mathbbm 1}_A((u_1,\tau_1),\dots,(u_{d+1},\tau_{d+1}))\,
\D(\tau_1,\dots,\tau_{d+1})\,\varphi^{d+1}(\D{\bf u}).
\end{eqnarray*}

Let $((u_1,\tau_1),\dots,(u_{d+1},\tau_{d+1}))\in A$. The hyperplanes $H(u_1,\tau_1),\dots,H(u_{d+1},\tau_{d+1})$ 
determine the simplex $\triangle$, with incenter $z$ and inradius $r$. Let $x_j\in L_{\bf u}$ be the point where $H(u_j,\tau_j)$ touches the inball of $\triangle$. Then
$$ \langle x_j-z,v_j/\|v_j\|\rangle = r$$
and
\begin{eqnarray*}
\tau_j &=& \langle x_j,u_j\rangle = \langle z,u_j\rangle + \langle x_j-z,u_j\rangle = \langle z,u_j\rangle + \langle x_j-z,v_j+w_j\rangle\\
&=& \langle z,u_j\rangle + \langle x_j-z,v_j\rangle =\langle z,u_j\rangle + r\|v_j\|\\
&=& t_j({\bf u}, z,r).
\end{eqnarray*}
Thus, we see that the mapping
$$ (z,r,{\bf u}) \mapsto ((u_1,t_1({\bf u},z,r)),\dots,(u_{d+1},t_{d+1}({\bf u},z,r)))$$
maps $\R^d\times\R_{>0}\times {\sf P}_k$ bijectively onto $A$. As mentioned, for fixed ${\bf u}\in{\sf P}_k$, the mapping 
$$ (z,r) \mapsto (t_1({\bf u},z,r),\dots,t_{d+1}({\bf u},z,r))$$ 
has Jacobian $D_k({\bf u})$. Therefore, the transformation formula for multiple integrals, applied for each fixed ${\bf u}\in{\sf P}_k$ to the inner integral, gives the assertion.
\end{proof}

\section{The distribution of the typical $k$-face}\label{sec4}

Our next aim is to extend \cite[Thm. 10.4.6]{SW08}, giving a representation for the distribution of the typical $k$-face $Z^{(k)}$ with respect to the incenter as a center function. 

For ${\bf u}\in {\sf P}_k$, we define
$$ T_k({\bf u}) := L_{\bf u}\cap\bigcap_{i=1}^{k+1}H^-(u_i,\|u_i|L_{\bf u}\|).$$
This is a $k$-dimensional simplex in the subspace $L_{\bf u}$, with incenter $o$ and inradius $1$.

For a $k$-dimensional affine subspace $L$, a point $z\in L$ and a number $r>0$, we define $B^o(L,z,r)$ as the relative interior (with respect to $L$) of the $k$-dimensional ball in $L$ with center $z$ and radius $r$.

Of the Poisson hyperplane process $\widehat X$ we assume that its spherical directional distribution is subspace-free. Then the hyperplanes of $\widehat X$ are a.s. in general position. For the typical $k$-face $Z^{(k)}$ of the tessellation $X$, we take the incenter as the center function $c$. This is possible by Lemmas \ref{L3.1} and \ref{L3.2}. Let $F\in X^{(k)}$ be a $k$-face of the tessellation $X$, and let $L$ be its affine hull. Then there are $d-k$ hyperplanes of $\widehat X$ whose intersection is $L$. They can be numbered in $(d-k)!$ ways. The face $F$ has a $k$-dimensional inball. It determines $k+1$ hyperplanes of $\widehat X$ whose intersections with $L$ touch the inball. These hyperplanes can be numbered in $(k+1)!$ ways. Therefore, (\ref{2.6}) gives, for $A\in\B(\K)$, the subsequent relation. In the following formulas, $L,\triangle, z,r$ are functions of the hyperplanes $H_1,\dots,H_{d+1}$ in general position (and the given number $k\in\{1,\dots,d-1\}$), namely
$$ L= \bigcap_{i=k+2}^{d+1} H_i,$$
$\triangle$ is the simplex in $L$ determined by $L\cap H_1,\dots,L\cap L_{k+1}$, the point $z$ is its incenter, and the number $r$ is its inradius. The dependence of the functions $L,\triangle, z,r$ on the hyperplanes $H_1,\dots,H_{d+1}$ is not shown in the following formulas, but has to be kept in mind. We recall that $H^-_x$ denotes the closed halfspace bounded by the hyperplane $H$ that contains $x$, provided that $x\notin H$.

We get by (\ref{2.6}), using the unit cube $C^d$ as the set $B$,
\begin{eqnarray*}
\gamma^{(k)}\bQ^{(k)}(A) &=& \bE \sum_{K\in X^{(k)}} {\mathbbm 1}(K-c(K)) {\mathbbm 1}_{C^d}(c(K))\\
&=& \frac{1}{(k+1)!(d-k)!} \,\bE\sum_{(H_1,\dots,H_{d+1})\in \widehat X^{d+1}_{\not=}} {\mathbbm 1} \left\{ \widehat X\cap[B^o(L,z,r)]_\Ha=\emptyset\right\} {\mathbbm 1}_{C^d}(z)\\
&& \times\, {\mathbbm 1}\left\{ \bigcap_{H\in\widehat X\setminus[B^o(L,z,r)]_\Ha} H_z^--z\in A\right\}.
\end{eqnarray*}
Application of the Mecke formula (Corollary 3.2.3 of \cite{SW08}) yields
\begin{eqnarray*}
\gamma^{(k)}\bQ^{(k)}(A) &=& \frac{1}{(k+1)!(d-k)!} \int_{\Ha^{d+1}} \bE{\mathbbm 1} \left\{ \widehat X\cap[B^o(L,z,r)]_\Ha=\emptyset\right\} {\mathbbm 1}_{C^d}(z)\\
&& \times\, {\mathbbm 1}\left\{ \bigcap_{H\in\widehat X\setminus[B^o(L,z,r)]_\Ha}\left(\triangle\cap H_z^-\right)-z\in A\right\} \widehat\Theta(\D(H_1,\dots,H_{d+1})).
\end{eqnarray*}

The processes $\widehat X\cap[B^o(L,z,r)]_\Ha$ and $\widehat X\setminus[B^o(L,z,r)]_\Ha$ are independent. Further, denoting by $L_o$ the subspace which is a translate of $L$ and by $B(L_o)$ the $k$-dimensional closed ball in $L_o$ with center $o$ and radius $1$, we have
$$
\bP\left(\widehat X\cap [B^o(L,z,r)]_\Ha =\emptyset\right) = \bP\left(\widehat X\cap[rB(L_o)]_\Ha=\emptyset\right)
= e^{-2\widehat\gamma r\Phi(B(L_o))}.
$$
Thus, we obtain
\begin{eqnarray*}
\gamma^{(k)}\bQ^{(k)}(A) &=& \frac{1}{(k+1)!(d-k)!} \int_{\Ha^{d+1}} e^{-2\widehat \gamma r\Phi(B(L_o))} {\mathbbm 1}_{C^d}(z)\\
&& \times\,\bP\left( \bigcap_{H\in\widehat X\setminus[B^o(L,z,r)]_\Ha}\left(\triangle\cap H_z^-\right)-z\in A\right)\widehat\Theta^{d+1}(\D(H_1,\dots,H_{d+1}))\\
&=&  \frac{1}{(k+1)!(d-k)!} \int_{\Ha^{d+1}} e^{-2\widehat\gamma r\Phi(B(L_o))} {\mathbbm 1}_{C^d}(z)\\
&& \times\,\bP\left( \bigcap_{H\in\widehat X\setminus[B^o(L_o,o,r)]_\Ha}\left((\triangle-z)\cap H_o^-\right)\in A\right)\widehat\Theta^{d+1}(\D(H_1,\dots,H_{d+1})),
\end{eqnarray*}
where in the last step the stationarity of $\widehat X$ was used. 

Now we apply Lemma \ref{L3.3}. We recall that $L,\triangle,z,r$ in the integrand above are functions of $H_1,\dots,H_{d+1}$. In particular, we have
$$ L_o(H(u_1,\tau_1),\dots,H(u_{d+1},\tau_{d+1}))=L_{\bf u},\quad {\bf u}=(u_1,\dots,u_{d+1}).$$
If we first denote (for fixed ${\bf u}=(u_1,\dots,u_{d+1})$) the integration variables of the inner integrals in Lemma \ref{L3.3} by $z'$ and $r'$, we have
$$ z(H(u_1,t_1({\bf u},z',r'),\dots)= z',$$
$$ r(H(u_1,t_1({\bf u},z',r'),\dots)= r',$$
$$ \triangle(H(u_1,t_1({\bf u},z',r'),\dots) -  z(H(u_1,t_1({\bf u},z',r'),\dots)= r'T_k({\bf u}).$$
Therefore, we obtain the following theorem.

\begin{theorem}\label{T4.1}
Suppose that the spherical directional distribution $\varphi$ of $\widehat X$ is subspace-free. Let $k\in\{1,\dots,d\}$. With respect to the incenter as center function, the distribution $\bQ^{(k)}$ of the typical $k$-face $Z^{(k)}$ of the tessellation $X$ induced by $\widehat X$ is given by
\begin{eqnarray*}
&&\bQ^{(k)}(A)\\
&& =\frac{2^{k+1}}{(k+1)!(d-k)!} \frac{\widehat\gamma^{d+1}}{\gamma^{(k)}} \int_{{\sf P}_k} \int_0^\infty e^{-2\widehat\gamma r\Phi(B(L_{\bf u}))} \,\bP \left(\bigcap_{H\in\widehat X \setminus [B^o(L_{\bf u},o,r)]_\Ha} H_o^-\cap r T_k({\bf u})\in A\right)\\
&& \hspace{4mm}\times\; \D r D_k({\bf u})\,\varphi^{d+1}(\D{\bf u})
\end{eqnarray*}
for $A\in\B(\K)$.
\end{theorem}

\section{The shape of small $k$-faces}\label{sec5}

As we want to find the limit distribution of the shape of small $k$-faces of the tessellation $X$, we need a measurement for size and a notion of shape. We consider $\cP_k$, the set of $k$-dimensional polytopes in $\R^d$, for a given number $k\in\{1,\dots,d\}$, and the subset  $\cP'_k$ of polytopes with a unique ($k$-dimensional) inball. A {\em size functional} for $\cP_k$ is a real function on $\K$ that is increasing under set inclusion, homogeneous of degree one, continuous (with respect to the Hausdorff metric), and positive on $k$-dimensional sets. The homogeneity of degree one is not a restriction against the formerly required homogeneity of some degree $m>0$, since a homogeneous functional $\psi$ of degree $m$ may be replaced by $\psi^{1/m}$. Examples of size functionals for $\cP_k$ are $V_k^{1/k}$, where $V_k$ is the $k$-dimensional volume, the mean width, the circumradius, or the diameter. On $\cP'_k$ we use the incenter as center function, denoted by $c$.

The {\em shape} of $P\in \cP'_k$, is obtained by normalization,
$$ \mathfrak{s}_c(P):= \frac{1}{\Phi(P)}(P-c(P)),$$
where $\Phi$ is the hitting functional defined in Section \ref{sec2}. The shape space $\cP_{k,c}$, defined by
$$ \cP_{k,c} = \mathfrak{s}_c(\cP'_k)=\{P\in\cP'_k: c(P)=o,\,\Phi(P)=1\},$$
is a Borel subset of $\cP$. By $f_r(P)$ we denote the number of $r$-dimensional faces of a polytope $P$.

\begin{theorem}\label{T5.1}
Suppose that the spherical directional distribution of $\widehat X$ is subspace-free. Let $k\in \{1,\dots,d\}$. Let $\Sigma$ be a translation invariant size functional for $\cP_k$. Then the typical cell $Z^{(k)}$ of the tessellation $X$ satisfies
$$ \lim_{a\to 0} \bP\left(\mathfrak{s}_c(Z^{(k)})\in S\mid \Sigma(Z^{(k)})<a\right) = \frac{\xi_{k,\varphi,\Sigma}(S)}{\xi_{k,\varphi,\Sigma}(\cP_{k,c})}$$
for $S\in\B(\cP_{k,c})$, where the measure $\xi_{k,\varphi,\Sigma}$ on $\B(\cP_{k,c})$ is defined by
$$ \xi_{k,\varphi,\Sigma} := \int_{{\sf P}_k} {\mathbbm 1}\{\mathfrak{s}_c(T_k({\bf u}))\in\cdot\} \frac{D_k({\bf u})}{\Sigma(T_k({\bf u}))}\,\varphi^{d+1}(\D{\bf u}).$$
In particular,
$$ \lim_{a\to 0} \bP\left(f_{k-1}(Z^{(k)})=k+1\mid \Sigma(Z^{(k)})<a\right)=1.$$
\end{theorem}

\begin{proof}
Let $S\in\B(\cP_{k,c})$ and $a>0$ be given. We apply Theorem \ref{T4.1} to the set
$$ A=\{P\in\cP'_k:f_{k-1}(P) =k+1,\,\mathfrak{s}_c(P)\in S,\,\Sigma(P)<a\}.$$
We have
\begin{eqnarray*} 
&& f_{k-1}\left(\bigcap_{H\in\widehat X\setminus[B^o(L_{\bf u},o,r)]_\Ha} H_o^-\cap rT_k({\bf u})=k+1\right) \\
&& \Leftrightarrow \widehat X\left([rT_k({\bf u})]_\Ha \setminus[B^o(L_{\bf u},o,r)]_\Ha\right)=0
\end{eqnarray*}
and
\begin{eqnarray*}
&& \bP\left(\widehat X([rT_k({\bf u})]_\Ha \setminus[B^o(L_{\bf u},o,r)]_\Ha)=0\right)\\
&&=  \bP\left(\widehat X([rT_k({\bf u})]_\Ha \setminus[rB(L_{\bf u})]_\Ha)=0\right)\\
&&= e^{-2\widehat\gamma r[\Phi(T_k({\bf u}))-\Phi(B(L_{\bf u}))]}.
\end{eqnarray*}
Hence, from Theorem \ref{T4.1}  we get
\begin{eqnarray*}
&& \bP\left(f_{k-1}(Z^{(k)})=k+1,\, \mathfrak{s}_c(Z^{(k)})\in S,\,\Sigma(Z^{(k)})<a\right)\\
&&=  \frac{2^{k+1}}{(k+1)!(d-k)!} \frac{\widehat\gamma^{d+1}}{\gamma^{(k)}} \int_{{\sf P}_k} \int_0^\infty e^{-2\widehat \gamma r\Phi(T_k({\bf u}))} \,{\mathbbm 1}\{\Sigma(rT_k({\bf u}))<a\}\,\D r\\
&& \hspace{4mm}\times\; {\mathbbm 1}\{\mathfrak{s}_c(T_k({\bf u}))\in S\} D_k({\bf u})\,\varphi^{d+1}(\D{\bf u}).
\end{eqnarray*}
By the mean value theorem, the inner integral is equal to
$$ \frac{a}{\Sigma(T_k({\bf u}))}e^{-2\widehat\gamma r_0({\bf u})\Phi(T_k({\bf u}))}$$
with an intermediate value $r_0({\bf u})$ satisfying
$$ 0\le r_0({\bf u}) \le \frac{a}{\Sigma(T_k({\bf u}))} \le \frac{a}{\Sigma(B(L_{\bf u}))} \le Ca,$$
wnere $C$ is a constant independent of ${\bf u}$. Here we have used that $B(L_{\bf u})\subset T_k({\bf u})$ and that $\Sigma$ is increasing, finally that $\Sigma$ is continuous, hence ${\bf u}\mapsto\Sigma(B(L_{\bf u}))$ attains a minimum, which is positive by the assumptions on the size functional. It follows that
\begin{eqnarray}
&& \lim_{a\to 0} a^{-1} \bP\left(f_{k-1}(Z^{(k)})=k+1,\,\mathfrak{s}_c(Z^{(k)})\in S, \, \Sigma(Z^{(k)})<a\right)\nonumber\\
&& =  \frac{2^{k+1}}{(k+1)!(d-k)!} \frac{\widehat\gamma^{d+1}}{\gamma^{(k)}} \int_{{\sf P}_k}  {\mathbbm 1}\{\mathfrak{s}_c(T_k({\bf u}))\in S\}\frac{ D_k({\bf u})}{\Sigma(T_k({\bf u}))}\,\varphi^{d+1}(\D{\bf u}). \label{5.1}
\end{eqnarray}

Let $r_k(K)$ denote the inradius of a $k$-dimensional convex body $K\subset L_{\bf u}$. Then $r_k(K)B(L_{\bf u})+t\subseteq K$ for a suitable vector $t$, hence $r_k(K)\Sigma(B(L_{\bf u}))\le\Sigma(K)$. Thus,
$$ \Sigma(Z^{(k)})<a\Rightarrow r_k(Z^{(k)}) <\frac{a}{\Sigma(B(L_{\bf u}))}.$$
Therefore,
$$ \bP\left(f_{k-1}(Z^{(k)})>k+1,\,\Sigma(Z^{(k)})<a\right) \le \bP\left(f_{k-1}(Z^{(k)})>k+1,\, r_k(Z^{(k)})<\frac{a}{\Sigma(B(L_{\bf u}))}\right).$$
With this, Theorem \ref{T4.1} gives
\begin{eqnarray*}
&&\bP\left(f_{k-1}(Z^{(k)})>k+1,\,\Sigma(Z^{(k)})<a\right)\\
&& \le \frac{2^{k+1}}{(k+1)!(d-k)!} \frac{\widehat\gamma^{d+1}}{\gamma^{(k)}} \int_{{\sf P}_k} \int_0^\infty e^{-2\widehat\gamma r\Phi(B(L_{\bf u}))}\\
&& \hspace{4mm}\times\;\bP \left(f_{k-1}\left(\bigcap_{H\in\widehat X \setminus [B^o(L_{\bf u},o,r)]_\Ha} H_o^-\cap r T_k({\bf u})\right) > k+1\right)\\
&& \hspace{4mm}\times\;{\mathbbm 1}\left\{r<\frac{a}{\Sigma(B(L_{\bf u}))}\right\}\,\D r D_k({\bf u})\,\varphi^{d+1}(\D{\bf u}). 
\end{eqnarray*}
For each fixed ${\bf u}\in{\sf P}_k$, the probability in the integrand tends to zero as $r\to 0$. Therefore, applying the mean value theorem as above, we now obtain
\begin{equation}\label{5.2}
\lim_{a\to 0} a^{-1}\bP\left(f_{k-1}(Z^{(k)})>k+1,\,\Sigma(Z^{(k)})<a\right)=0.
\end{equation}
From relations (\ref{5.1}) and (\ref{5.2}), the assertions of Theorem \ref{T5.1} follow immediately.
\end{proof}

\noindent Author's address:\\[2mm]Rolf Schneider\\Mathematisches Institut, Albert-Ludwigs-Universit{\"a}t\\D-79104 Freiburg i. Br., Germany\\E-mail: rolf.schneider@math.uni-freiburg.de


\begin{thebibliography}{10}

\bibitem{BRT14} Beermann, M., Redenbach, C., Th\"ale, C., Asymptotic shape of small cells. {\em Math. Nachr.} {\bf 287} (2014), 737--747.

\bibitem{Bon17} Bonnet, G., Small cells in a Poisson hyperplane tessellation. {\em Adv. Appl. Math.} {\bf 95} (2018), 31--52.

\bibitem{HRS04a} Hug, D., Reitzner, M., Schneider, R.,  The limit shape of the zero cell in a stationary Poisson hyperplane tessellation.  \textit{Ann. Probab.} \textbf{32} (2004), 1140--1167.

\bibitem{HRS04b} Hug, D., Reitzner, M., Schneider, R., Large Poisson--Voronoi cells and Crofton cells. \textit{Adv. Appl. Prob. (SGSA)} \textbf{36} (2004), 667--690.

\bibitem{HS07a} Hug, D., Schneider, R., Asymptotic shapes of large cells in random tessellations. \textit{Geom. Funct. Anal.} \textbf{17} (2007), 156--191.
 
\bibitem{HS07b} Hug, D., Schneider, R., Typical cells in Poisson hyperplane tessellations. \textit{Discrete Comput. Geom.} \textbf{38} (2007), 305--319.  

\bibitem{HS10} Hug, D., Schneider, R.,  Large faces in Poisson hyperplane mosaics. {\em Ann. Probab.} {\bf 38} (2010), 1320--1344.

\bibitem{HS11} Hug, D., Schneider, R.,  Faces with given directions in anisotropic Poisson hyperplane mosaics. {\em Adv. Appl. Prob.} {\bf 43} (2011), 308--321.

\bibitem{Kov97} Kovalenko, I.N.,  A proof of a conjecture of David Kendall on the shape of random polygons of large area. (Russian) \textit{Kibernet. Sistem. Anal.} 1997, 3--10, 187; Engl. transl. \textit{Cybernet. Systems Anal. } \textbf{33} (1997), 461--467.

\bibitem{Kov99} Kovalenko, I.N.,  A simplified proof of a conjecture of D.G. Kendall concerning shapes of random polygons. \textit{J. Appl. Math. Stochastic Anal.} \textbf{12} (1999), 301--310.

\bibitem{Mil95} Miles, R.E., A heuristic proof of a long-standing conjecture of D.G. Kendall concerning the shapes of certain large random polygons. {\em Adv. Appl. Prob. (SGSA)} {\bf 27} (1995), 397--417.

\bibitem{RS16} Reitzner, M., Schneider, R., On the cells in a stationary Poisson hyperplane mosaic. {\em Adv. Geom.}, DOI 10.1515/advgeom-2018-0013; arXiv:1609.04230 (2016).

\bibitem{Sch18} Schneider, R., The polytopes in a Poisson hyperplane tessellation. (submitted), arXiv:1804.05622v1 (2018).

\bibitem{SW08} Schneider, R., Weil, W., {\em Stochastic and Integral Geometry.} Springer, Berlin, 2008.

\end{thebibliography}
\end{document}